\documentclass{conm-p-l}
\usepackage{appendix}
\usepackage[usenames]{color}
\usepackage{amsthm}
\usepackage{amssymb}
\usepackage{amsmath}
\usepackage{textcomp, yfonts,mathrsfs,ulem, wasysym, pxfonts, pifont, bbm, graphicx}
\usepackage[all, dvips]{xy}
\usepackage{verbatim,rotating,tikz}
\usepackage[pdfpagemode=FullScreen,bookmarks=true]{hyperref}
\usepackage{setspace}

\newtheorem{thm}{Theorem}[section]
\newtheorem*{thm*}{Theorem}
\newtheorem{lem}[thm]{Lemma}

\newtheorem*{defn*}{Definition}
\newtheorem{prop}[thm]{Proposition}

\newtheorem{rem}[thm]{Remark}

\newtheorem{conj}[thm]{Conjecture}

\theoremstyle{remark}

\newcommand{\Z}{\mathbb{Z}}

\newcommand{\ra}{\rightarrow}

\newcommand{\Map}{\text{Map}}

\renewcommand{\Top}{\text{Top}}

\DeclareMathOperator{\holim}{holim}

\DeclareMathOperator{\hocolim}{hocolim}
\newcommand{\0}{\emptyset}
\newcommand{\D}{\mathscr{D}}
\newcommand{\X}{\mathcal{X}}

\begin{document}

\title{Absolutely homotopy-cartesian squares}


%    author one information
\author{Rosona Eldred}
\address{Department of Mathematics\\
Universit\"{a}t Hamburg\\
Bundesstr. 55\\
D- 20146 Hamburg\\
Germany
}
\curraddr{}
\email{rosona.eldred@math.uni-hamburg.de}
\thanks{}

%    The 2010 edition of the Mathematics Subject Classification is
%    now available.  If you are citing a classification from the
%    new scheme, use the following input coding instead.
%\subjclass[2010]{Primary }

\date{}

\begin{abstract}
We call a diagram $\D$ absolutely cartesian if $F(\D)$ is homotopy cartesian for all homotopy functors $F$. This is a sensible notion for diagrams in categories $C$ where Goodwillie's calculus of functors may be set up for functors with domain $C$. %In particular, $C$ may be spaces or spectra. % and functors $F$ that go from spaces/spectra to spaces/spectra. 
We prove a classification theorem for absolutely cartesian squares of spaces and state a conjecture of the classification for higher dimensional cubes.
\end{abstract}
\maketitle
%\chapter{Absolutely cartesian squares}\label{ch:abscart}
\bigskip

Let $I$ be a small indexing category with initial object $\emptyset$ and final object 1. A diagram $\D$ in a category $C$ is a functor $I \ra C$; we restrict ourselves here to $C$ being spaces.% or spectra. 
This diagram is cartesian \footnote{In keeping with conventions of Goodwillie's calculus of functors, we only deal with homotopy cartesian diagrams, so omit the ``homotopy'' modifier.} 
when $\D (\emptyset )$ is equivalent to the homotopy limit of $\D$ over $I$ with $\emptyset$ removed, denoted $\holim_{I_\0} \D$ or $\holim_{\emptyset} \D$ when $I$ is clear from context. Similarly, $\D$ is cocartesian if $\D(1)$ is equivalent to the homotopy colimit over $I$ with the final object removed, denoted $\hocolim_{I^1} \D$; as in the cartesian case, the $I$ subscript is omitted if clear from context and we write $\hocolim_1$. A functor $F$ is a homotopy functor if it is weak-equivalence-preserving. We call a diagram $\D$ %, a functor from a small indexing category $I$ to a (simplicial model) category $C$, 
absolutely (co)cartesian if $F(\D)$ is homotopy (co)cartesian for all homotopy functors $F$. Note that a diagram is an $(n+1)$ cube if it is indexed by $I=\mathscr{P}([n])$, the powerset on $[n]=\{0,1,\ldots n\}$.  %This is a sensible notion for diagrams in spaces or spectra and functors $F$ that go from spaces/spectra to spaces/spectra. 

%%%%%%%%%%%%%%%%%%%%%%%%%%%%%%%%%%%%%%%%%%%%%%%%%%%%%%%%%%%%
%%%%%%%%%%%%%%%%%%%%%%%%%%%%%%%%%%%%%%%%%%%%%%%%%%%%%%%%%%%%
\section{Statements of Results and Conjectures}
%%%%%%%%%%%%%%%%%%%%%%%%%%%%%%%%%%%%%%%%%%%%%%%%%%%%%%%%%%%%
%%%%%%%%%%%%%%%%%%%%%%%%%%%%%%%%%%%%%%%%%%%%%%%%%%%%%%%%%%%%

We prove the following classification theorem for absolutely cartesian squares:% of spaces: 

%%%%%%%%%%%%%%%%%
% THEOREM
%%%%%%%%%%%%%%%%%
\begin{thm}\label{thm:abscartsq}
A square of spaces %or spectra a, $\mathscr{D}: \mathscr{P}([1]) \ra \Top$ (or Sp), 
is absolutely cartesian if and only if it is a map of two absolutely cartesan 1-cubes. That is, of the following form (the other two maps may also be equivalences): 
\[
\xymatrix{
A \ar[r]^{\sim}\ar[d] & B\ar[d]\\
C \ar[r]^{\sim} & D\\
}
\]
\end{thm}
%%%%%%%%%%%%%%%%%
%%%%%%%%%%%%%%%%%
\medskip

%A pushout/pullback of spectra is levelwise a pushout/pullback of its constituent spaces. To prove the conjecture for spectra, it suffices to prove it for spaces. 
Theorem \ref{thm:abscartsq} is the base case of our following conjecture:

%%%%%%%%%%%%%%%%%
%  REVISED CONJECTURE
%%%%%%%%%%%%%%%%%
\begin{conj}\label{conj1}
An $n$-cube of spaces %or spectra 
is absolutely cartesian if and only if it can be written as either a map of two absolutely cartesian $(n-1)$-cubes or a chain of compositions of $n$-cubes of these types. \footnote{There should be some way to express this as the cubes being ``generated by'' those built out of absolutely cartesian squares.}
\end{conj}
%%%%%%%%%%%%%%%%%
%%%%%%%%%%%%%%%%%
%The base case of this, for $n=2$, is the main theorem of this section, Theorem .  
It should be clear that building up an $n$-cube inductively as maps of these absolutely cartesian squares  and compositions of such cubes will yield an absolutely cartesian $n$-cube, which is the $\Leftarrow$ direction of the if and only if. To be clear, two cubes $\mathscr{C,D}$ may be composed if they can be written $\mathscr{C}: X\ra Y$ and $\D:Y \ra Z$; their composition is then $\mathscr{C}\circ \D :X \ra Z$. Geometrically, this looks like ``glueing'' the cubes along their shared face. We give an example in the next section. By chain of compositions, we mean compositions of possibly more than two cubes, e.g.  $\mathscr{C}\circ \D\circ \mathscr{E}$ where $\mathscr{C,D,E}$ are all $n$-cubes built inductively up from maps of absolutely cartesian squares. 

It is not yet certain if the other direction is true. We may observe that the absolutely cartesian squares are also absolutely cocartesian. Thus, we make an additional conjecture that 

%%%%%%%%%%%%%%%%%
% CONJECTURE
%%%%%%%%%%%%%%%%%
\begin{conj}\label{conj2}
An $n$-cube is absolutely cartesian if and only if it is absolutely cocartesian.
\end{conj}
%%%%%%%%%%%%%%%%%
%%%%%%%%%%%%%%%%%

If we include contravariant functors, we can show this conjecture for $n=2$, and we will comment on this after the proof for cartesian squares, which is in the following section. 

We will present partial results towards Conjecture \ref{conj2} in section \ref{sec:conj2}; this includes a positive verification of the conjecture when restricting to functors which land in 1-connected spaces (including the identity implies that the spaces in the diagram must originally be 1-connected as well).

The section after that is about a family of 3-cubes which are absolutely cocartesian and cartesian and which are not expressible as a map of two absolute cartesian squares, but as a composition of 3-cubes of that form. We end with applications and related work. 

%%%%%%%%%%%%%%%%%%%%%%%%%%%%%%%%%%%%%%%%%%%%%%%%%%%%%%%%%%%%
%%%%%%%%%%%%%%%%%%%%%%%%%%%%%%%%%%%%%%%%%%%%%%%%%%%%%%%%%%%%
\section{Proof of Classification for Squares}
%%%%%%%%%%%%%%%%%%%%%%%%%%%%%%%%%%%%%%%%%%%%%%%%%%%%%%%%%%%%
%%%%%%%%%%%%%%%%%%%%%%%%%%%%%%%%%%%%%%%%%%%%%%%%%%%%%%%%%%%%

%%%%%%%%%%%%%%%%%
% PROOF
%%%%%%%%%%%%%%%%%
\begin{proof}[Proof of Theorem \ref{thm:abscartsq}]\footnote{The current form (and brevity) of this proof is influenced heavily by conversations between the author and Tom Goodwillie about developing a clearer route towards attacking the more general conjecture.}%. We had chatted for a few days at a conference about redoing this proof in a way that would lend a clearer route towards attacking the more general conjecture.}%(due to Tom Goodwillie)
This relies on switching briefly to the setting of spectra and using this to deduce properties of the original diagram of spaces. We also point out that it suffices to prove that either $B \ra D$ or $C \ra D$ is an equivalence, since equivalences are stable under homotopy pullback. That is, it implies that the mirroring map, $A\ra C$ or $A\ra B$, is also an equivalence.

Consider an absolutely cartesian square of spaces:
\[
\xymatrix{
A \ar[r] \ar[d]& B \ar[d]\\
C \ar[r] & D\\
}
\]

Now apply the functor $\Sigma^{\infty} \Map(D,-)$ to our square: 
 \[
\xymatrix{
\Sigma^{\infty}\Map(D, A) \ar[r] \ar[d]& \Sigma^{\infty}\Map(D,B) \ar[d]\\
\Sigma^{\infty}\Map(D, C) \ar[r] & \Sigma^{\infty}\Map(D,D)\\
}
\]

By assumption, this resultant square is still cartesian. Since the square is in spectra, we know that it is also cocartesian. Recall that $\Sigma^{\infty}$ commutes with colimits. 

We then have the following chain of equivalences:
\medskip
\[
\begin{array}{ccc}
\pi_0 \Sigma^{\infty} \hocolim (\Map(D,B) \leftarrow \Map(D,A) \ra \Map (D,C) &\simeq& \pi_0 \Sigma^{\infty} \Map(D,D).\\
\parallel & & \parallel\\
H_0  (\hocolim (\Map(D,B) \leftarrow \Map(D,A) \ra \Map (D,C) ) & & H_0 (\Sigma^{\infty}\Map(D,D))\\
\parallel & & \parallel\\
\Z [ \pi_0 (\hocolim (\Map(D,B) \leftarrow \Map(D,A) \ra \Map (D,C)) ] & & \Z[\pi_0 \Map(D,D)]\\
\end{array}
\]
\medskip

We can interpret this as telling us that $(\pi_0 \Map (D,B) \cup \pi_0 \Map (D,C))/\sim$ surjects onto $\pi_0 \Map (D,D)$.  Consider $id \in \Map(D,D)$.  This then has a preimage (up to homotopy) in $\Map(D,B)$ and/or $\Map(D,C)$; assume $\Map(D,B)$.  This gives a section $D \ra B$. We can then rewrite our original diagram with our new map in the pre-image of the identity. This is Figure \ref{fig:setup}.

\begin{figure}[h]
\[
%\scalebox{.8}{$
\xymatrix{
                    & D \ar[d] \ar@/^/[dd]^{id}\\
A \ar[r] \ar[d]& B \ar[d]\\
C \ar[r]          & D\\
}%$}
\]
\caption{New information included in diagram}
\label{fig:setup}
\end{figure}
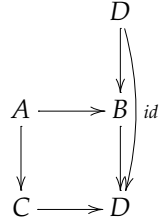

We can add the homotopy pullback of $(A \ra B \leftarrow D)$ to the diagram. Then the whole diagram is a pullback, being a composition of pullback squares. This lets us pull back the identity map, as in Figure \ref{fig:pullbackid}.

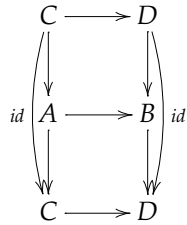
\begin{figure}[h!]
\[
%\scalebox{.8}{$
\xymatrix{
 \ar@/_/[dd]_{id}C\ar[r]\ar[d]  & D \ar[d] \ar@/^/[dd]^{id}\\
A \ar[r] \ar[d]& B \ar[d]\\
C \ar[r]          & D\\
}
%$}
\]
\caption{Adding the pullback of the top punctured square and pulling back the identity}
\label{fig:pullbackid}
\end{figure}

The whole diagram is itself absolutely cartesian (having two facing maps which are equivalences). Since the bottom and entire squares are both absolutely cartesian, so is the top square, shown again in Figure \ref{fig:top1}.
\begin{figure}[h!]
\[
%\scalebox{.8}{$
\xymatrix{
 %\ar@/_/[dd]_{id}
 C\ar[r]\ar[d]  & D \ar[d]\\% \ar@/^/[dd]^{id}\\
A \ar[r] & B\\% \ar[d]\\
%C \ar[r]          & D\\
}
%$}
\]
\caption{``top'' square}
\label{fig:top1}
\end{figure}

Now that the top square is known to be absolutely cartesian, we can proceed in the same as we did with the original square, and obtain a section from $B$ to $D$ or $A$. If the section is to $D$, we are done, as we already have a splitting from $D$ to $B$ and having another the other direction gives us an equivalence between $B$ and $D$. 

Otherwise, we work in the other direction. We add our section $B \ra A$ to our diagram, in Figure \ref{fig:wBsec}, shown without the other equivalences.

\begin{figure}[h]
\[
\xymatrix{
 & C \ar[d]\ar[r] & D\ar[d]\\
 B \ar@/_/[rr] \ar[r] & A\ar[d] \ar[r] & B\ar[d]\\
 & C \ar[r] & D\\
}
\]
\caption{Adding the section $B\ra A$}
\label{fig:wBsec}
\end{figure}
  
Then we pull back the upper left square. The square comprised of the upper left and right squares together is then a cartesian square, with bottom map an equivalence. These are stable under pullback, meaning that the identity map $B\ra B$ is pulled back, this time to the ``top''.  Thus we know the pullback of the left square is equivalent to $D$, so the top two squares are as in Figure \ref{fig:pbBD}. This implies that the left square is also absolutely cartesian, as the entire and the right ones are.

\begin{figure}[h]
 \[
\xymatrix{
 D\ar[r]\ar[d]\ar@/^/[rr]^{id}& C \ar[d]\ar[r] & D\ar[d]\\
 B \ar[r]\ar@/_/[rr]_{id}& A \ar[r] & B\\
}
\]
\caption{Pulled back identity to top of diagram}
\label{fig:pbBD}
\end{figure}

%Then since the right square and the whole square are both absolutely cartesian, so is the lefthand square. 

Then we return to $A$, and the (now) absolutely cartesian square in the left of Figure \ref{fig:pbBD}. In the same way as before, we get a section $A \ra C$ or $A\ra B$.  If $A\ra B$ is a section, we are done -- the one-sided inverse (our original section) has an inverse on the other side and $A\overset{\simeq}{\ra} B$, which implies immediately that $D\simeq C$ since it occurs in the cartesian diagram on the left side of Figure \ref{fig:pbBD} .   %Then since we have the original ``top'' square cartesian,  as in Figure \ref{fig:topcart}, we pull back the equivalence between $A$ and $B$ to conclude that $C \overset{\simeq}{\ra} D$ as well. 

If $A \ra C$  instead is a section, we are done. This is because the map $C\ra A$ was a section obtained earlier. We conclude $A\overset{\simeq}{\ra} C$, which implies immediately that $D\simeq B$ since it occurs in the cartesian diagram on the left side of Figure \ref{fig:pbBD} and equivalences are pulled back.

%%%%%%%%
\if false
\begin{figure}[h]
\[
\xymatrix{
C \ar[r]\ar[d] & D\ar[d]\\
A \ar[r]^{\sim} & B\\
}
\]
\caption{Two splittings giving our equivalence}
\label{fig:topcart}
\end{figure}
%Cartesianness implies $C \overset{\simeq}{\ra} D$.

%Recall the ``upper left'' square, shown in Figure \ref{fig:ulcart}. Since it's also cartesian, we can pullback the equivalence and conclude that  $D\overset{\simeq}{\ra} B$. Since this map was a section (one-sided inverse) to $B \ra D$, we can conclude that our original map $B\overset{\simeq}{\ra} D$ was an equivalence.  

\begin{figure}[h!]
\[
\xymatrix{
D \ar[r] \ar[d]& C\ar[d]^{\sim}\\
B \ar[r] & A
}
\]
\caption{Upper Left cartesian square}
\label{fig:ulcart}
\end{figure}
\fi

%Cartesianness lets us conclude that $D\overset{\simeq}{\ra} B$. Since this map was a section (one-sided inverse) to $B \ra D$, we can conclude that our original map $B\overset{\simeq}{\ra} D$ was an equivalence.  

\end{proof}

\begin{rem}
For absolutely \textbf{cocartesian} squares, if we allow our homotopy functors to also be possibly contravariant, then we can establish that they are of the same form as absolutely cartesian squares. The proof is parallel to that for cartesian squares, with the functor $\Sigma^\infty \Map(D,-)$ replaced by $\Sigma^\infty \Map(-, A)$.  	
\end{rem}

%%%%%%%%%%%%%%%%%%%%%%%%%%%%%%%%%%%%%%%%%%%%%%%%%%%%%%%%%%%%
%%%%%%%%%%%%%%%%%%%%%%%%%%%%%%%%%%%%%%%%%%%%%%%%%%%%%%%%%%%%
\section{Partial results for Conjecture \ref{conj2}}\label{sec:conj2}
%%%%%%%%%%%%%%%%%%%%%%%%%%%%%%%%%%%%%%%%%%%%%%%%%%%%%%%%%%%%
%%%%%%%%%%%%%%%%%%%%%%%%%%%%%%%%%%%%%%%%%%%%%%%%%%%%%%%%%%%%

As observed by the anonymous reviewer, the following weakened form of the conjecture already holds: 
%%%%%%%%%%%%%%%%%
%%%%%%%%%%%%%%%%%
\begin{prop}\label{prop:rev}
If one restricts to $n$-cubes of 1-connected spaces and homotopy functors which take values in 1-connected-spaces, Conjecture \ref{conj2} holds. The direction (absolutely cocartesian implies absolutely cartesian) holds in the weaker condition of nilpotent\footnote{A space $X$ is nilpotent when $\pi_1(X)$ is a nilpotent group. 1-connected spaces are trivially nilpotent.} spaces and functors taking values in nilpotent spaces.  
\end{prop}

\begin{proof}[Proof]This is also following the reviewer.\\
\begin{enumerate} \itemsep 5pt
\item \textbf{Functors with nilpotent target, abs cocartesian $\Rightarrow$ abs cartesian}. Let $\X$ be absolutely cocartesian. The functor $\Sigma^\infty$ from Spaces to Spectra preserves the cocartesianness, and in Spectra, diagrams are cocartesian iff cartesian. $\Omega^\infty$ from Spectra to Spaces preserves cartesianness, so $Q\X := \Omega^\infty \Sigma^\infty \X$ is cartesian, in addition to remaining cocartesian. Repeated applications of $Q$ will clearly retain this property; that is, $QQ\cdots Q\X$ will be cartesian and cocartesian. 

As $\Omega^\infty$ and $\Sigma^\infty$ are an adjoint pair and $Q$ the associated monad\footnote{Also referred to as a "triple".}, there is an associated cosimplicial ``$Q$-completion" (a.k.a $\Z$-nilpotent completion) for any space.  For a space $X$, the $Q$-completion of $X$ is the homotopy limit of the cosimplicial space which arises naturally from iterating the monadic maps $X \ra Q(X)$ and $QQX \ra QX$. 
\[
\xymatrix{QX \ar@<3pt>[r]\ar@<-3pt>[r]& \ar[l] QQX \ar[r]\ar@<6pt>[r]\ar@<-6pt>[r]&  \ar@<-3pt>[l]\ar@<3pt>[l]\cdots 
}
\]
The same line of reasoning holds with $\X$ replaced by $F(\X)$ (since $F(\X)$ is also absolutely cocartesian), so $\X$ is absolutely cartesian. 

\item  \textbf{Functors with 1-connected target, absolutely cartesian $\Rightarrow$ absolutely cocartesian}.  Let $\X$ be absolutely cartesian. Then $F(\X)$ and $\Sigma^\infty F(\X)$ for all hofunctors $F: \Top \ra \Top$ are also cartesian; in particular, $\Sigma^\infty F(\X)$ is also cocartesian. Since $F$ takes values in 1-connected spaces, this is sufficient to conclude that $F(\X)$ itself is cocartesian. This is for all hofunctors $F$, so $\X$ is absolutely cocartesian.  
\end{enumerate}
\end{proof}

%%%%%%%%%%%%%%%%%
%%%%%%%%%%%%%%%%%

It was also pointed out by Goodwillie in a discussion with the author that 
%%%%%%%%%%%%%%%%%
%%%%%%%%%%%%%%%%%
\begin{prop}\label{prop:tom}
If (absolutely cocartesian $\Rightarrow$ absolutely cartesian), then (absolutely cartesian $\Rightarrow$ absolutely cocartesian). 
\end{prop}
%%%%%%%%%%%%%%%%%
%%%%%%%%%%%%%%%%%

\begin{proof}[\textbf{Proof Sketch:}]  Let $\X$ be absolutely cartesian. Then for all $F,G$ hofunctors and $A, B$ in the appropriate categories,  $\Map (F(\Map (G(\X), A)), B)$ is also cartesian. Unwrapping the dependencies and keeping in mind that $\Map( -, Y)$ takes cocartesian to cartesian, we get that $\Map(G(\X),A)$ is absolutely cocartesian.  Apply our hypothesis and that $\Map( -, Y)$ takes cocartesian to cartesian to conclude that $\X$ is also absolutely cocartesian.  
\end{proof} %\begin{flushright}$\qed$\end{flushright}

%%%%%%%%%%%%%%%%%%%%%%%%%%%%%%%%%%%%%%%%%%%%%%%%%%%%%%%%%%%%
%%%%%%%%%%%%%%%%%%%%%%%%%%%%%%%%%%%%%%%%%%%%%%%%%%%%%%%%%%%%
\section{An Absolutely Cocartesian and Cartesian 3-cube}
%%%%%%%%%%%%%%%%%%%%%%%%%%%%%%%%%%%%%%%%%%%%%%%%%%%%%%%%%%%%
%%%%%%%%%%%%%%%%%%%%%%%%%%%%%%%%%%%%%%%%%%%%%%%%%%%%%%%%%%%%
 %would explain that the example in Section 3 shows that the obvious attempt to extend Theorem 1.1 to cubes of higher dimension fails. 
 The original form of Conjecture \ref{conj1} was as follows:
 \begin{quote}
 \textit{An $(n+1)$-cube of spaces $\X$ is absolutely cartesian iff there are absolutely cartesian $n$-cubes $Y, Z$ such that $\X: Y \ra Z$.}
 \end{quote}
\medskip

This was corrected to the current form of the conjecture due to the following illustrative example; a cube which may be expressed as the \textit{composition} of two cubes of the aforementioned type without being a map of two such absolutely (co)cartesian squares.  

Given maps $A\ra B\ra D\ra B\ra C$ with the condition that $B\ra D \ra B$ is equivalent to the identity, the following 3-cube may be assembled: \footnote{We thank the referee for this example, which made it clear that we need to include not just \textit{maps} of $(n-1)$ cubes.}: 
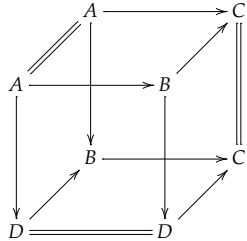
 \[
 \scalebox{.75}{$
 \xymatrix{
                             &A \ar[rr] \ar[dd]& & C\ar@{=}[dd] \\ 
   A \ar@{=}[ur] \ar[rr] \ar[dd]&                  &  B \ar[dd]\ar[ur]\\
                               &B \ar[rr]& & C \\
   D \ar@{=}[rr]\ar[ur]  &                          & D\ar[ur]\\
 }$}
 \]
Now that we have more complicated diagrams, we have chosen to denote equivalences by equality so that it is clear which maps are equivalences. 
 
It is possible to first establish absolute cartesianness and cocartesianness independent of the decomposition, but this is superfluous once we have the decomposition. We will then just provide the decomposition. 
 \if false
We will first establish absolute cartesianness and cocartesianness independent of 
the decomposition, and then give the decomposition.

%%%%%%%%%%%%%%%%%%%%%%%%%%%%%%%%%%%%%%%%%%%%%%%%%%%%%%%%%%%%
\subsection{Estabilishing Absolute (co)Cartesianness}
%%%%%%%%%%%%%%%%%%%%%%%%%%%%%%%%%%%%%%%%%%%%%%%%%%%%%%%%%%%%
While an $n$-cube is cartesian if and only if its $(n-1)$ cubes of fibers (for all choices of point to fiber over) is cartesian, there is only a partial dual to this statement. If we know an $n$-cube is cocartesian, then it will follow that its $(n-1)$ cofiber cubes will also be cocartesian. \textit{However}, there are cubes with cocartesian cofiber cubes which are not themselves cocartesian. This boils down to the difference between cartesian squares and cocartesian squares; the former induces a long exact sequence in homotopy groups and the latter only in (co)homology groups. 

We are not totally bereft of tools to check for cocartesianness. What is perhaps the most common is the ``Covering Lemma"\cite[Prop 0.2 and its dual]{GC2}, which we will state a restricted version of for cubical diagrams:

\begin{lem}\label{lem:cover}
Let $\D$ be an $n$-cubical diagram (of spaces or spectra). It may be expressed as a map of $(n-1)$ cubes in $n$ ways. For any such expression, $\D: X \ra Y$, $\D$ is cartesian precisely when 
\[
\xymatrix{
X_\emptyset \ar[r]\ar[d] & \holim_\emptyset X\ar[d]\\
Y_\emptyset \ar[r] & \holim_\emptyset Y
}
\]
is cartesian; $X_\emptyset$ denotes the initial object of the diagram $X$ (and likewise for $Y$) . $\D$ is cocartesian precisely when 
\[
\xymatrix{
\hocolim_1 X \ar[r] \ar[d] & X_{[n-1]}\ar[d]\\
\hocolim_1 Y \ar[r] & Y_{[n-1]}\\
}
\]
is cocartesian; here, $X_{[n-1]}$ is the final element of the diagram $X$ (and similarly for $Y$).
\end{lem}

This lemma allows us to determine cartesianness (and cocartesianness) by considering it as a map of two squares $X \ra Y$ (possibly in two ways, one for cartesianness and one for cocartesianness). 
\medskip 

For cartesianness, we consider the 3-cube as a map 
\[
\xymatrix{
A \ar[r] \ar[d] & B \ar[d]\\
D \ar[r]^{\simeq} & D\\
}
\hspace{1 cm}
\Rightarrow
\hspace{1 cm}
\xymatrix{
A \ar[r] \ar[d] & C \ar[d]^{\simeq}\\
B \ar[r]& C\\
}.
\]

Let $X$ denote the source square and $Y$ the target square. The 3-cube is cartesian then if the following
\[
\xymatrix{
A=X_\emptyset \ar[r]\ar[d] & (\holim_\emptyset X) \simeq B  \ar[d]\\
A=Y_\emptyset \ar[r] & (\holim_\emptyset Y) \simeq B\\
}
\]
is cartesian. The property that composition the maps $B \ra D \ra B$ is the identity is what gives us that the induced map $B\simeq(\holim_\emptyset X) \ra (\holim_\emptyset Y) \simeq B$ is also equivalent to the identity. It is clear that this is true for any homotopy functor $F$ we apply, except the square will be of $F(A)$ and $F(B)$. 
\medskip 

For cocartesianness, we consider the 3-cube as a map 
\[
\xymatrix{
A \ar[r] \ar[d]^{\simeq} & C \ar[d]\\
A \ar[r]& B\\
}
\hspace{1 cm}
\Rightarrow
\hspace{1 cm}
\xymatrix{
B \ar[r] \ar[d] & C \ar[d]\\
D \ar[r]^{\simeq}& D\\
}.
\]

Let $W$ denote the source square and $Z$ the target square. The 3-cube is cartesian then if the following
\[
\xymatrix{
(\hocolim_1 W) \simeq B  \ar[d]\ar[r] & W_{[1]}=C\ar[d]\\
(\hocolim_1 Z) \simeq B \ar[r]& Z_{[1]}=C\\
}
\]
is cocartesian. In this case, the maps $B \ra D \ra B$ compsing to the identity is what gives us that the map $B\simeq (\hocolim_1 W)\ra(\hocolim_1 Z) \simeq B$ is an equivalence. It is again clear that this is true for any homotopy functor $F$ we apply, except the square will be of $F(B)$ and $F(C)$. 

This 3-cube is not written as a map of two absolutely (co)cartesian squares, but it has just been show to be absolutely cocartesian and cartesian. 
\fi 

%%%%%%%%%%%%%%%%%%%%%%%%%%%%%%%%%%%%%%%%%%%%%%%%%%%%%%%%%%%%
\subsection{Factorization}
%%%%%%%%%%%%%%%%%%%%%%%%%%%%%%%%%%%%%%%%%%%%%%%%%%%%%%%%%%%%
Despite not being a map of two absolutely (co)cartesian squares, the 3-cube \footnote{This factorization related to one pointed out by Tom Goodwillie.} may be expressed as a composition of two 3-cubes which \textit{are} of that form. This relies on the ability to express $B$ as a retract of $D$. We compose the cube 
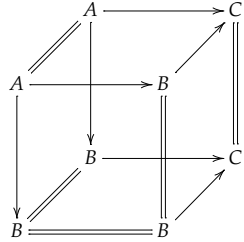
\[ \scalebox{.75}{$
\xymatrix{
                             &A \ar[rr] \ar[dd]& & C\ar@{=}[dd] \\ 
   A \ar@{=}[ur] \ar[rr] \ar[dd]&                  &  B \ar@{=}[dd]\ar[ur]\\
                                &B \ar[rr]& & C\\
   B \ar@{=}[rr]\ar@{=}[ur]  &                          & B\ar[ur]\\  
   %
                           %    &B \ar[rr]& & C \\
 %  D \ar@{=}[rr]\ar[ur]  &                          & D\ar[ur]\\
 }$}
\]
with 
\[ \scalebox{.75}{$
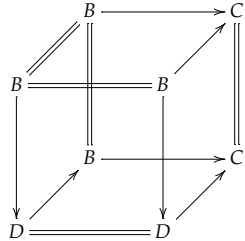
\xymatrix{
 %                            &A \ar[rr] \ar[dd]& & C\ar@{=}[dd] \\ 
  % A \ar@{=}[ur] \ar[rr] \ar[dd]&                  &  B \ar@{=}[dd]\ar[ur]\\
   %
                                &B \ar[rr]\ar@{=}[dd]& & C\ar@{=}[dd] \\
   B \ar@{=}[rr]\ar@{=}[ur]\ar[dd]  &                          & B\ar[dd]\ar[ur]\\  
                               &B \ar[rr]& & C \\
   D \ar@{=}[rr]\ar[ur]  &                          & D\ar[ur]\\
 }$}
\]
to get our original 3 cube as the total cube of the composition (`glueing' the first cube atop the second):

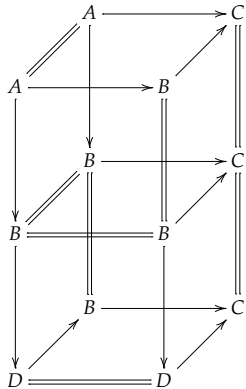
\[ \scalebox{.75}{$
\xymatrix{
                             &A \ar[rr] \ar[dd]& & C\ar@{=}[dd] \\ 
   A \ar@{=}[ur] \ar[rr] \ar[dd]&                  &  B \ar@{=}[dd]\ar[ur]\\
                                &B \ar[rr]\ar@{=}[dd]& & C\ar@{=}[dd] \\
   B \ar@{=}[rr]\ar@{=}[ur]\ar[dd]  &                          & B\ar[dd]\ar[ur]\\  
                               &B \ar[rr]& & C \\
   D \ar@{=}[rr]\ar[ur]  &                          & D\ar[ur]\\
 }$}
\]

\if false
 and if the following is cartesian
\[
\xymatrix{

}
\]
then the 3-cube is cartesian. If the following is cocartesian, 
\[

\]
the 3-cube is cocartesian.
\fi

\section{Applications and Related Work}
%[
We end with a few remarks on extending this and other approaches. %]
Par\'{e}\cite{pare} studies strict colimits which are preserved by all functors, and calls such colimits absolute, a naming convention which we have chosen to follow by calling our \textit{homotopy} diagrams absolute when preserved by all homotopy functors.  Street works in an enriched setting and states his results in terms of distributors\cite{street}. It is not clear at the moment how applicable their results are in this setting. The first step would be to switch to considering simplicial functors, which are (roughly) as good as homotopy functors, to work enriched. 

The original goal to classifying absolutely cartesian cubes was to get ``wrong way'' maps, from holims of cubes of one dimension to ones of a higher dimension, in a certain diagram related to the $E_1$ page of the spectral sequence associated to a cosimplicial space. These are going the wrong way inasmuch as natural maps between diagrams are usually from lower to higher dimension, which induces a map from the holim of the higher dimensional diagram to the holim of the lower dimensional diagram.

 A map of cubes  \textit{of the same dimension}, $A \ra B$, induces maps on the homotopy limits of the cubes $\holim A \ra \holim B$ (also for the punctured homotopy limits, $\holim_\emptyset$). If A and B are diagrams, with B cartesian, $\holim_0 (A \ra B) \simeq \holim_\emptyset A$.  That is, a way to take an $n$ cube and produce an $(n+1)$ cube with equivalent (punctured) homotopy limit is to find a cartesian $n$ cube to which it maps naturally. We would also like do these constructions only once for all homotopy functors, so the cube we are mapping to not only needs to be cartesian, but with cartesianness preserved by all homotopy functors.   
% allows us to work enriched. %We anticipate that it is possible to take (cite Pare) and (cite Street) 

%% Finish 
%\if  false
%%%%%%%%%%%%%%%%%%%%%%%%%%%%%%%
% BIBLIOGRAPHY
\bibliographystyle{amsplain}
%\bibliography{arolla.bib}
\providecommand{\bysame}{\leavevmode\hbox to3em{\hrulefill}\thinspace}
\providecommand{\MR}{\relax\ifhmode\unskip\space\fi MR }
\providecommand{\MRhref}[2]{%
  \href{http://www.ams.org/mathscinet-getitem?mr=#1}{#2}
}
\providecommand{\href}[2]{#2}

%%%%%%%%%%%%%%%%%%%%%%%%%%%%%%%%%%%%%%%%%%
%%%%%%%%%%%%%%%%%%%%%%%%%%%%%%%%%%%%%%%%%%
\end{document}